\documentclass{amsart}
\usepackage{amsfonts,amssymb, wasysym}
\usepackage{multicol}
\usepackage{hyperref}
\hypersetup{colorlinks,citecolor=blue,linkcolor=blue}
\newtheorem{proposition}{Proposition}
\newtheorem{theorem}{Theorem}

\newcommand{\integers}{{\mathbb Z}}
\newcommand{\realnos}{{\mathbb R}}

\usepackage{hyperref}
\hypersetup{colorlinks,citecolor=blue,linkcolor=blue}

\begin{document}

\title{Hyperbolic 24-Cell 4-Manifolds with One Cusp}

\author{John G. Ratcliffe and Steven T. Tschantz}

\address{Department of Mathematics, Vanderbilt University, Nashville, TN 37240
\vspace{.1in}}

\email{j.g.ratcliffe@vanderbilt.edu}

\date{}

\begin{abstract}
In this paper, we describe all the hyperbolic 24-cell 4-manifolds with exactly one cusp. 
There are four of these manifolds up to isometry. 
These manifolds are the first examples of one-cusped hyperbolic 4-manifolds of minimum volume. 
\end{abstract}

\maketitle

\section{Introduction}\label{S:1} 
The 24-cell is a regular 4-dimensional polytope in either Euclidean, Spherical or hyperbolic 4-space 
with exactly 24 sides each of which is a regular octahedron. 
A {\it hyperbolic 24-cell manifold} is a hyperbolic 4-manifold that is obtained 
from an ideal, regular, hyperbolic 24-cell by gluing each side to another side by an isometry. 
Examples of hyperbolic 24-cell manifolds are given in our paper \cite{R-T-V}.
Hyperbolic 24-cell manifolds have minimum volume among hyperbolic 4-manifolds.  
In this paper, all hyperbolic manifolds are assumed to be complete. 
As a reference for hyperbolic manifolds, see \cite{R}.

The main result of this paper is the following 
classification of all one-cusped hyperbolic 24-cell manifolds: 

\begin{theorem}\label{T1}
There are exactly four hyperbolic 24-cell manifolds, with a single cusp, up to isometry. 
Each of these manifolds is non-orientable. The link of each cusp is affinely equivalent to the second non-orientable 
closed flat 3-manifold $N^3_2$ in the Hantzsche-Wendt \cite{H-W} classification of closed flat 3-manifolds. 
\end{theorem}

The volume of a hyperbolic 4-manifold $M$ of finite volume is given by the formula 
$$\mathrm{Vol}(M) = \textstyle{\frac{4}{3}}\pi^2 \chi(M).$$ 
Our examples are the first known examples of one-cusped hyperbolic 4-manifolds 
of Euler characteristic 1, and therefore of minimum volume.

The first examples of one-cusped hyperbolic 4-manifold were constructed by Kolpakov and Martelli \cite{K-M}, 
and in particular, they found an example with $\chi = 4$.  Examples of one-cusped hyperbolic 4-manifolds with $\chi = 2$ were constructed by Slavich and Kolpakov \cite{S, K-S-S, K-S-C}. 
The existence of our examples answers Question 4.17 of \cite{K-S-S} in the affirmative. 

Our paper is organized as follows:  In \S2, the flat manifold $N^3_2$ is described. 
In \S3, the classification of the one-cusped hyperbolic 24-cell manifolds is described. 
In \S4, we discuss how to obtain a presentation for the fundamental group of a hyperbolic 24-cell manifold. 
In \S5, \ldots, \S8, we give a presentation for the fundamental group of each of the one-cusped hyperbolic 24-cell manifolds. 
In \S9, we determine the volume of the maximum cusp of each of the one-cusped hyperbolic 24-cell manifolds. 
In \S10, we determine the order of the group of isometries of each of the one-cusped hyperbolic 24-cell manifolds. 
In \S11, we describe the orientable double covers of each of the one-cusped hyperbolic 24-cell manifolds. 

\section{The Closed Flat 3-Manifold $N^3_2$} 

The non-orientable closed flat 3-manifolds are classified up to affine equivalence by their first homology group, with $H_1(N^3_2) = \integers \oplus \integers$.  The flat 3-manifold $N^3_2$ is the Euclidean space-form $E^3/\Gamma$ 
where $\Gamma$ is the crystallographic group with International Tables Number 9. 
For the standard affine representation of $\Gamma$, see Table 1B of \cite{B-Z}. 
The flat 3-manifold $N^3_2$ fibers in various ways (see Table 1 of \cite{R-T-F}). 
In particular $N^3_2$ is a torus bundle over the circle with monodromy 
the orientation-reversing isometry of the flat torus $E^2/\integers^2$ induced by 
the reflection of $E^2$ defined by the matrix
$$\left(\begin{array}{rr} 0 & 1 \\ 1 & 0 \end{array}\right).$$

\section{Classification of the One-Cusped Hyperbolic 24-Cell Manifolds} 

To prove Theorem \ref{T1}, we first found all the possible side-pairings of an ideal, regular, hyperbolic 24-cell 
that yield a hyperbolic 4-manifold by an exhaustive computer search over a search space of order $(23!!) 48^{12}$. 
Our search was successful because the ridge and edge cycle conditions, defining when a side-pairing gives a manifold, 
each depend on only a few of the side-pairing maps.  A backtracking search in the space of partial side-pairings that incrementally checks these conditions can eliminate potential side-pairings early on.  
The next side to extend a partial side-pairing is chosen strategically
to maximize the conditions that can be checked thus limiting the number of search tree branches that need to be examined.
We further took advantage of the symmetry of the regular 24-cell. 

Up to symmetry of the regular 24-cell, we found 13,108 side-pairings that yield 
a hyperbolic 4-manifold. Only four of these side-pairings yield a manifold with a single cusp. 
We classified these one-cusped manifolds up to isometry by computing their homology groups 
using a cellular homology chain complex. 
The one-cusped hyperbolic 24-cell manifolds are classified by their first homology groups (see Table \ref{Ta1}). 

\begin{table}  
\begin{tabular}{llllll}
Name & $H_0$ & $H_1$ & $H_2$ & $H_3$ &  $H_4$ \\
\hline 
24c1.1  & $\integers$ & $\integers\oplus\integers_3\oplus\integers_{13}$  & $\integers$ & $0$  & $0$ \\ 
24c1.2  & $\integers$ & $\integers\oplus\integers_3$  & $\integers\oplus\integers_{13}$ & $0$  & $0$ \\ 
24c1.3  & $\integers$ & $\integers_2^3\oplus \integers_5$ & $\integers_2$ & $0$ & $0$  \\
24c1.4  & $\integers$ & $\integers_2^2\oplus\integers_3^3$  & $\integers_3$ & $0$  & $0$ \end{tabular}

\medskip
\caption{Homology groups of the one-cusped hyperbolic 24-cell manifolds}\label{Ta1}
\end{table}

We now set up notation to describe the side-pairings for the four one-cusped hyperbolic 24-cell manifolds. 
The hyperboloid model of hyperbolic 4-space is 
$$H^4 = \{x \in \realnos^5: x_1^2 + \cdots + x_4^2 - x_5^2 = -1\ \hbox{and}\ x_5 > 0 \}.$$
We identify the group of isometries of $H^4$ with the positive Lorentz group $\mathrm{O}^+(4,1)$. 

We will work with a standard 24-cell $Q$ in $H^4$ 
that is centered at the central point $(0, 0, 0, 0, 1)$ of $H^4$. 
The outward normal vectors to the sides of $Q$, 
with respect to the Lorentzian inner product 
$$x \circ  y = x_1y_1 +\cdots + x_4y_4 - x_5y_5,$$ 
are taken in the following fixed order:

$$\begin{array}{rlrl}
s_1&\!\!=\ (1,1,0,0,1)&
s_{13}&\!\!=\ (1,0,0,1,1)\\
s_2&\!\!=\ (-1,1,0,0,1)&
s_{14}&\!\!=\ (-1,0,0,1,1)\\
s_3&\!\!=\ (1,-1,0,0,1)&
s_{15}&\!\!=\ (1,0,0,-1,1)\\
s_4&\!\!=\ (-1,-1,0,0,1)&
s_{16}&\!\!=\ (-1,0,0,-1,1)\\
s_5&\!\!=\ (1,0,1,0,1)&
s_{17}&\!\!=\ (0,1,0,1,1)\\
s_6&\!\!=\ (-1,0,1,0,1)&
s_{18}&\!\!=\ (0,-1,0,1,1)\\
s_7&\!\!=\ (1,0,-1,0,1)&
s_{19}&\!\!=\ (0,1,0,-1,1)\\
s_8&\!\!=\ (-1,0,-1,0,1)&
s_{20}&\!\!=\ (0,-1,0,-1,1)\\
s_9&\!\!=\ (0,1,1,0,1)&
s_{21}&\!\!=\ (0,0,1,1,1)\\
s_{10}&\!\!=\ (0,-1,1,0,1)&
s_{22}&\!\!=\ (0,0,-1,1,1)\\
s_{11}&\!\!=\ (0,1,-1,0,1)&
s_{23}&\!\!=\ (0,0,1,-1,1)\\
s_{12}&\!\!=\ (0,-1,-1,0,1)&
s_{24}&\!\!=\ (0,0,-1,-1,1)
\end{array}$$

The 24-cell is self-dual, and so $Q$ has 24 ideal vertices, which are taken in the following fixed order:

$$\begin{array}{rlrl}
v_1&\!\!=\ (-\frac{1}{2},-\frac{1}{2},-\frac{1}{2},-\frac{1}{2},1)&
v_{13}&\!\!=\ (\frac{1}{2},\frac{1}{2},-\frac{1}{2},-\frac{1}{2},1) \vspace{.05in}\\
v_2&\!\!=\ (-\frac{1}{2},-\frac{1}{2},-\frac{1}{2},\frac{1}{2},1)&
v_{14}&\!\!=\ (\frac{1}{2},\frac{1}{2},-\frac{1}{2},\frac{1}{2},1) \vspace{.05in}\\
v_3&\!\!=\ (-\frac{1}{2},-\frac{1}{2},\frac{1}{2},-\frac{1}{2},1)&
v_{15}&\!\!=\ (\frac{1}{2},\frac{1}{2},\frac{1}{2},-\frac{1}{2},1) \vspace{.05in}\\
v_4&\!\!=\ (-\frac{1}{2},-\frac{1}{2},\frac{1}{2},\frac{1}{2},1)&
v_{16}&\!\!=\ (\frac{1}{2},\frac{1}{2},\frac{1}{2},\frac{1}{2},1) \vspace{.05in}\\
v_5&\!\!=\ (-\frac{1}{2},\frac{1}{2},-\frac{1}{2},-\frac{1}{2},1)&
v_{17}&\!\!=\ (1,0,0,0,1)\vspace{.05in}\\
v_6&\!\!=\ (-\frac{1}{2},\frac{1}{2},-\frac{1}{2},\frac{1}{2},1)&
v_{18}&\!\!=\ (-1,0,0,0,1)\vspace{.05in}\\
v_7&\!\!=\ (-\frac{1}{2},\frac{1}{2},\frac{1}{2},-\frac{1}{2},1)&
v_{19}&\!\!=\ (0,1,0,0,1)\vspace{.05in}\\
v_8&\!\!=\ (-\frac{1}{2},\frac{1}{2},\frac{1}{2},\frac{1}{2},1)&
v_{20}&\!\!=\ (0,-1,0,0,1)\vspace{.05in}\\
v_9&\!\!=\ (\frac{1}{2},-\frac{1}{2},-\frac{1}{2},-\frac{1}{2},1)&
v_{21}&\!\!=\ (0,0,1,0,1)\vspace{.05in}\\
v_{10}&\!\!=\ (\frac{1}{2},-\frac{1}{2},-\frac{1}{2},\frac{1}{2},1)&
v_{22}&\!\!=\ (0,0,-1,0,1)\vspace{.05in}\\\
v_{11}&\!\!=\ (\frac{1}{2},-\frac{1}{2},\frac{1}{2},-\frac{1}{2},1)&
v_{23}&\!\!=\ (0,0,0,1,1)\vspace{.05in}\\\
v_{12}&\!\!=\ (\frac{1}{2},-\frac{1}{2},\frac{1}{2},\frac{1}{2},1)&
v_{24}&\!\!=\ (0,0,0,-1,1)
\end{array}$$

\vspace{.15in}
We next list the sets of indices of the vertices of each of the 24 sides of $Q$.

$$\begin{array}{rlrl}
S_1&\!\!=\ \{13, 14, 15, 16, 19\}&
S_{13}&\!\!=\ \{10, 12, 14, 16, 17, 23\}\\
S_2&\!\!=\ \{5, 6, 7, 8, 18, 19\}&
S_{14}&\!\!=\ \{2, 4, 6, 8, 18, 23\}\\
S_3&\!\!=\ \{9, 10, 11, 12, 17, 20\}&
S_{15}&\!\!=\ \{9, 11, 13, 15, 17, 24\}\\
S_4&\!\!=\ \{1, 2, 3, 4, 18, 20\}&
S_{16}&\!\!=\ \{1, 3, 5, 7, 18, 24\}\\
S_5&\!\!=\ \{11, 12, 15, 16, 17, 21\}&
S_{17}&\!\!=\ \{6, 8, 14, 16, 19, 23\}\\
S_6&\!\!=\ \{3, 4, 7, 8, 18, 21\}&
S_{18}&\!\!=\ \{2, 4, 10, 12, 20, 23\}\\
S_7&\!\!=\ \{9, 10, 13, 14, 17, 22\}&
S_{19}&\!\!=\ \{5, 7, 13, 15, 19, 24\}\\
S_8&\!\!=\ \{1, 2, 5, 6,18, 22\}&
S_{20}&\!\!=\ \{1, 3, 9, 11, 20, 24\}\\
S_9&\!\!=\ \{7, 8, 15, 16, 19, 21\}&
S_{21}&\!\!=\ \{4, 8, 12, 16, 21, 23\}\\
S_{10}&\!\!=\ \{3, 4, 11, 12, 20, 21\}&
S_{22}&\!\!=\ \{2, 6, 10, 14, 22, 23\}\\
S_{11}&\!\!=\ \{5, 6, 13, 14, 19, 22\}&
S_{23}&\!\!=\ \{3, 7, 11, 15, 21, 24\}\\
S_{12}&\!\!=\ \{1, 2, 9, 10, 20, 22\}&
S_{24}&\!\!=\ \{1, 5, 9, 13, 22, 24\}
\end{array}$$

\newpage
The side-pairing for 24c1.1 is given in terms of indices of vertices as follows: 

$$\begin{array}{lrrrrrr}
S_1\to S_5:  & 13 \to 21, & 14 \to 11, & 15 \to 16, &16 \to 17, &17 \to 12, &19 \to 15 \\
S_2\to S_9:  &  5 \to 21,  & 6 \to 15,   &  7 \to 8,   & 8 \to 19, & 18 \to 7, &19 \to 16 \\
S_3\to S_{12}:  & 9 \to 20, & 10 \to 9, & 11 \to 2, & 12 \to 22, & 17 \to 10, & 20 \to 1 \\
S_4\to S_8: &1 \to 18, & 2 \to 1, & 3 \to 6, & 4 \to 22, & 18 \to 5, & 20 \to 2 \\
S_6\to S_{21}:  &3 \to 21, & 4 \to 8, & 7 \to 12, & 8 \to 23, & 18 \to 16, & 21 \to 4 \\
S_7\to S_{24}:  & 9 \to 24, & 10 \to 5, &13 \to 9, &14 \to 22, &17 \to 1, & 22 \to 13 \\
S_{10}\to S_{18}: & 3 \to 23, & 4 \to 2, & 11 \to 12, & 12 \to 20, & 20 \to 4, & 21 \to 10 \\
S_{11}\to S_{19}: & 5 \to 19, & 6 \to 5,  & 13 \to 15, & 14 \to 24, & 19 \to 13, & 22 \to 7 \\
S_{13}\to S_{17}: & 10 \to 14, & 12 \to 19, & 14 \to 23, &16 \to 8, & 17 \to 6, & 23 \to 16 \\
S_{14}\to S_{22}: & 2 \to 6, & 4 \to 23, & 6 \to 22,  & 8 \to 10, & 18 \to 14, & 23 \to 2 \\
S_{15}\to S_{23}: & 9 \to 7, & 11 \to 21, & 13 \to 24, & 15 \to 11, & 17 \to 3, & 24 \to 15 \\
S_{16}\to S_{20}: & 1 \to 9, & 3 \to 24, & 5 \to 20, & 7 \to 3, & 18 \to 11, & 24 \to 1
\end{array}$$

Note that we list only the side-pairings $S_i \to S_j$ with $i < j$, 
since $S_j \to S_i$ is obtained from $S_i \to S_j$ by reversing arrows and reordering 
so that initial indices are in increasing order.

\vspace{.15in}
The side-pairing for  24c1.2 is given in terms of indices of vertices as follows: 

$$\begin{array}{lrrrrrr}
S_1\to S_6: & 13 \to 18, & 14 \to 4, & 15 \to 7, & 16 \to 21, & 17 \to 8, & 19 \to 3 \\
S_2\to S_{10}: & 5 \to 20, & 6 \to 12, & 7 \to 3, & 8 \to 21, & 18 \to 11, & 19 \to 4 \\
S_3\to S_{11}: & 9 \to 22, & 10 \to 14, & 11 \to 5, & 12 \to 19, & 17 \to 6, & 20 \to 13 \\
S_4\to S_7: & 1 \to 22, & 2 \to 10, & 3 \to 13, & 4 \to 17, & 18 \to 9, & 20 \to 14\\
S_5\to S_{24}: & 11 \to 9, & 12 \to 24, & 15 \to 22, & 16 \to 5, & 17 \to 13, & 21 \to 1 \\
S_8\to S_{21}: & 1 \to 12, & 2 \to 21, & 5 \to 23, & 6 \to 8, & 18 \to 4, & 22 \to 16 \\
S_9\to S_{20}: & 7 \to 9, & 8 \to 20, & 15 \to 24, & 16 \to 3, & 19 \to 1, & 21 \to 11 \\
S_{12}\to S_{17}: & 1 \to 14, & 2 \to 23, & 9 \to 19, & 10 \to 8, & 20 \to 16, & 22 \to 6 \\
S_{13}\to S_{23}: & 10 \to 21, & 12 \to 11, & 14 \to 7, & 16 \to 24, & 17 \to 15, & 23 \to 3 \\
S_{14}\to S_{19}: & 2 \to 24, & 4 \to 5, & 6 \to 15, & 8 \to 19, & 18 \to 7, & 23 \to 13 \\
S_{15}\to S_{18}: & 9 \to 20, & 11 \to 2, &13 \to 12, & 15 \to 23, & 17 \to 10, & 24 \to 4 \\
S_{16}\to S_{22}: & 1 \to 23, & 3 \to 10, & 5 \to 6, & 7 \to 22, & 18 \to 2, & 24 \to 14
\end{array}$$

\vspace{.15in}
The side-pairing for  24c1.3 is given in terms of indices of vertices as follows:

$$\begin{array}{lrrrrrr}
S_1\to S_5: & 13\to17, & 14\to 16, & 15 \to 11, & 16 \to 21, & 17 \to 12, & 19 \to 15 \\
S_2\to S_{13}: & 5 \to 23, & 6 \to 12,  & 7 \to 14,  &  8 \to 17,  &  18 \to 16, & 19 \to 10 \\
S_3\to S_{16}: & 9 \to 5,   &10 \to 24, & 11 \to 18, & 12 \to 3,  & 17 \to 7,      & 20 \to 1 \\
S_4 \to S_8:   & 1 \to 6,    & 2  \to 22,  & 3  \to 18,  & 4 \to 1,   & 18 \to 2,    & 20 \to  5 \\
S_6 \to S_{21}: &3 \to 23, & 4 \to 16,  & 7 \to 4,   & 8 \to 21,   &  18 \to 8,    & 21 \to 12 \\
S_7 \to S_{23}:  & 9 \to 21 & 10 \to 3, & 13 \to 15, & 14 \to 24, & 17 \to 11, & 22 \to 7 \\
S_9 \to S_{18}:  & 7 \to 12, & 8 \to 20, &15 \to 23, & 16 \to 2, & 19 \to 4, & 21 \to 10 \\
S_{10}\to S_{20}: & 3 \to 20, & 4 \to 3, &11 \to 9, & 12 \to 24, & 20 \to 11, & 21 \to 1 \\
S_{11}\to S_{24}:  & 5 \to 13, & 6 \to 22, &13 \to 24, &14 \to 1, & 19 \to 9, & 22 \to 5 \\
S_{12}\to S_{19}:  & 1 \to 5,  & 2 \to 13, &  9 \to 7,  & 10 \to 15, & 20 \to 24, & 22 \to 19 \\
S_{14}\to S_{15}:  & 2 \to 9,  & 4 \to 13,  & 6 \to 11, & 8 \to 15, & 18 \to 24,  & 23 \to 17 \\
S_{17}\to S_{22}:  & 6 \to 10 & 8 \to 22,  & 14 \to  23, & 16 \to 6, & 19 \to 14, & 23 \to 2
\end{array}$$

The side-pairing for  24c1.4 is given in terms of indices of vertices as follows:

$$\begin{array}{lrrrrrr}
S_1\to S_9:  &13 \to 21, & 14 \to 15, & 15 \to 8, & 16 \to 19, & 17 \to 16, & 19 \to 7 \\
S_2\to S_{19}: & 5 \to 13, & 6 \to 15, & 7 \to 5, & 8 \to 7, & 18 \to 19, & 19 \to 24 \\
S_3\to S_{18}: & 9 \to 10, & 10 \to 12, & 11 \to 2, & 12 \to 4, & 17 \to 20, & 20 \to 23 \\
S_4\to S_{12}: & 1 \to 20, & 2 \to 9, & 3 \to 2, & 4 \to 22, & 18 \to 1, & 20 \to 10 \\
S_5\to S_{23}: &11 \to 3, & 12 \to 7, & 15 \to 11, & 16 \to 15, & 17 \to 21, & 21 \to 24 \\
S_6\to S_7: & 3 \to 9, & 4 \to 17, & 7 \to 22, & 8 \to 14, & 18 \to 13, & 21 \to 10 \\
S_8\to S_{22}: & 1 \to 2, & 2 \to 6, & 5 \to 10, & 6 \to 14, & 18 \to 22, & 22 \to 23 \\
S_{10}\to S_{21}: & 3 \to 12, & 4 \to 23, & 11 \to 21, & 12 \to 8, & 20 \to 16, & 21 \to 4 \\
S_{11}\to, S_{24}: & 5 \to 9, & 6 \to 22, & 13 \to 24, & 14 \to 5, & 19 \to 1, & 22 \to 13 \\
S_{13}\to S_{16}: & 10 \to 3, & 12 \to 1, & 14 \to 7, & 16 \to 5, & 17 \to 24, & 23 \to 18 \\
S_{14}\to S_{15}: & 2 \to 13, & 4 \to 15, & 6 \to 9, & 8 \to 11, & 18 \to 24, & 23 \to 17 \\
S_{17}\to S_{20}: & 6 \to 20, & 8 \to 3, & 14 \to 9, & 16 \to 24, & 19 \to 11, & 23 \to 1
\end{array}$$

\section{Presentations for the corresponding Discrete Groups} 

In this section, we discuss how to obtain a presentation for the fundamental group of a hyperbolic 24-cell manifold  
from a side-pairing of the 24-cell $Q$ together with an ordering of the sides of $Q$. 

A side-pairing map $S_i \to S_j$ determines a side-pairing transformation $g_i$ in $\mathrm{O}^+(4,1)$ 
which is the composition $r_jf_i$ of the symmetry $f_i$ of $Q$ that corresponds to the side-pairing map $S_i \to S_j$ 
followed by the reflection $r_j$ of $H^4$ in the side $S_j$. 
Note that $g_j = (g_i)^{-1}$, since $(g_i)^{-1} = (f_i^{-1}r_jf_i)f_i^{-1} = r_i f_j$, and so we will assume that $i < j$.

By Poincar\'e's fundamental polyhedron theorem, a side-pairing of $Q$ together with an ordering of the sides of $Q$  determines a set of 12 side-pairing transformations 
that form a set of generators for a discrete subgroup $\Gamma_\ast$ of $\mathrm{O}^+(4,1)$ 
whose orbit space $H^4/\Gamma_\ast$ is isometric to the hyperbolic 4-manifold $M$ obtained by gluing together the sides of $Q$ by the side-pairing.  
The fundamental group of $M$ is isomorphic to $\Gamma_\ast$. 

The dihedral angles of the regular polytope $Q$ are all $\pi/2$. 
Therefore $Q$ determines a regular tessellation $\mathcal{Q}$ of $H^4$ with fundamental cell $Q$. 
The group of symmetries of the tesselation is a $(3, 4, 3, 4)$ Coxeter simplex reflection group $\Gamma$ 
generated by the reflections of $H^4$ represented by the following matrices in $\mathrm{O}^+(4,1)$:

$$\frac{1}{2}\left(\begin{array}{rrrrr}
1 & 1 & 1 & 1 & 0 \\
1 & 1 & -1 & -1 & 0 \\
1 & -1 & 1 & -1 & 0 \\
1 & -1 & -1 & 1 & 0 \\ 
0 & 0 & 0 & 0 & 2
\end{array}\right),
\left(\begin{array}{rrrrr}
1 & 0 & 0 & 0 & 0 \\
0 & 1 & 0 & 0 & 0 \\
0 & 0 & 1 & 0 & 0 \\
0 & 0 & 0 & -1 & 0 \\ 
0 & 0 & 0 & 0 & 1
\end{array}\right), 
\left(\begin{array}{rrrrr}
1 & 0 & 0 & 0 & 0 \\
0 & 1 & 0 & 0 & 0 \\
0 & 0 & 0 & 1 & 0 \\
0 & 0 & 1 & 0 & 0 \\ 
0 & 0 & 0 & 0 & 1
\end{array}\right),
$$
$$\left(\begin{array}{rrrrr}
1 & 0 & 0 & 0 & 0 \\
0 & 0 & 1 & 0 & 0 \\
0 & 1 & 0 & 0 & 0 \\
0 & 0 & 0 & 1 & 0 \\ 
0 & 0 & 0 & 0 & 1
\end{array}\right),
\left(\begin{array}{rrrrr}
-1 & -2 & 0 & 0 & 2 \\
-2 & -1 & 0 & 0 & 2\\
0 & 0 & 1 & 0 & 0 \\
0 & 0 & 0 & 1 & 0 \\ 
-2 &-2 & 0 & 0 & 3
\end{array}\right). 
$$
The last matrix represents the reflection of $H^4$ in Side 1 of $Q$. 
The first four matrices generate the group of symmetries of $Q$, 
which has order 1152. The group $\Gamma_\ast$ is a torsion-free subgroup of $\Gamma$ of index 1152. 

By Poincar\'e's fundamental polyhedron theorem, defining relators for the side-pairing transformation generators of $\Gamma_\ast$ are in one-to-one correspondence  with ridge cycles determined by the side-pairing.  A ridge is a co-dimension 2 face. The 24-cell $Q$ has 96 ridges, each an ideal triangle, that are partitioned into cycles of order 4, since the dihedral angles of $Q$ are all $\pi/2$. Therefore, there are exactly 24 ridge cycles, and 24 corresponding defining relators 
of length 4 for the group $\Gamma_\ast$. 

Let $\Gamma_k$ be the torsion-free subgroup of $\Gamma$ of index 1152 determined by 
the given side-pairing for the manifold 24c1.{\it k} for $k = 1, \ldots, 4$.

\section{Presentation for the Group $\Gamma_1$} 

The 12 side-pairing transformations 
$g_1, g_2, g_3, g_4, g_6, g_7, g_{10}, g_{11}, g_{13}, g_{14}, g_{15}, g_{16}$
that generate the group $\Gamma_1$ are represented in $\mathrm{O}^+(4,1)$ as follows:

$$\begin{array}{ll}
g_1 = \frac{1}{2}\left(\begin{array}{rrrrr}
-3 & -3 & 1 & 1 & 4 \\
-1 & 1 & 1 & -1 & 0 \\
-3 & -3 & -1 & -1 & 4 \\
1 & -1 & 1 & -1 & 0 \\ 
-4 & -4 & 0 & 0 & 6
\end{array}\right),&
g_2 = \frac{1}{2}\left(\begin{array}{rrrrr}
1 & 1 & -1 & 1 & 0 \\
3 & -3 & 1 & 1 & 4 \\
3 & -3 & -1 & -1 & 4 \\
1 & 1 & 1 & -1 & 0 \\ 
4 & -4 & 0 & 0 & 6
\end{array}\right), \\ \\
g_3 = \frac{1}{2}\left(\begin{array}{rrrrr}
1 & 1 & -1 & 1 & 0 \\
3 & -3 & 1 & 1 & -4 \\
3 & -3 & -1 & -1 & -4 \\
1 & 1 & 1 & -1 & 0 \\ 
-4 & 4 & 0 & 0 & 6
\end{array}\right),&
g_4 = \frac{1}{2}\left(\begin{array}{rrrrr}
-3 & -3 & 1 & 1 & -4 \\
-1 & 1 & 1 & -1 & 0 \\
-3 & -3 & -1 & -1 & -4 \\
1 & -1 & 1 & -1 & 0 \\ 
4 & 4 & 0 & 0 & 6
\end{array}\right), \\ \\
g_6 = \frac{1}{2}\left(\begin{array}{rrrrr}
-1 & 1 & -1 & -1 & 0 \\
-1 & -1 & -1 & 1 & 0 \\
3 & -1 & -3 & -1 & 4 \\
3 & 1 & -3 & 1 & 4 \\ 
4 & 0 & -4 & 0 & 6
\end{array}\right), &
g_7 = \frac{1}{2}\left(\begin{array}{rrrrr}
-1 & 1 & -1 & -1 & 0 \\
-1 & -1 & -1 & 1 & 0 \\
3 & -1 & -3 & -1 & -4 \\
3 & 1 & -3 & 1 & -4 \\ 
-4 & 0 & 4 & 0 & 6
\end{array}\right), \\ \\
g_{10} = \frac{1}{2}\left(\begin{array}{rrrrr}
1 & 1 & 1 & -1 & 0 \\
-1 & -3 & 3 & -1 & -4 \\
1 & -1 & -1 & -1 & 0 \\
-1 & 3 & -3 & -1 & 4 \\ 
0 & 4 & -4 & 0 & 6
\end{array}\right), &
g_{11} = \frac{1}{2}\left(\begin{array}{rrrrr}
1 & 1 & 1 & -1 & 0 \\
-1 & -3 & 3 & -1 & 4 \\
1 & -1 & -1 & -1 & 0 \\
-1 & 3 & -3 & -1 & -4 \\ 
0 & -4 & 4 & 0 & 6
\end{array}\right), \\ \\
g_{13} = \frac{1}{2}\left(\begin{array}{rrrrr}
-1 & -1 & -1 & 1 & 0 \\
-3 & -1 & 1 & -3 & 4 \\
-1 & 1 & 1 & 1 & 0 \\
-3 & 1 & -1 & -3 & 4 \\ 
-4 & 0 & 0 & -4 & 6
\end{array}\right), &
g_{14} = \frac{1}{2}\left(\begin{array}{rrrrr}
-1 & 1 & 1 & -1 & 0 \\
-1 & -1 & -1 & -1 & 0 \\
-3 & -1 & 1 & 3 & -4 \\
3 & -1 & 1 & -3 & 4 \\ 
4 & 0 & 0 & -4 & 6
\end{array}\right), \\ \\
g_{15} = \frac{1}{2}\left(\begin{array}{rrrrr}
-1 & 1 & 1 & -1 & 0 \\
-1 & -1 & -1 & -1 & 0 \\
-3 & -1 & 1 & 3 & 4 \\
3 & -1 & 1 & -3 & -4 \\ 
-4 & 0 & 0 & 4 & 6
\end{array}\right), &  
g_{16} = \frac{1}{2}\left(\begin{array}{rrrrr}
-1 & -1 & -1 & 1 & 0 \\
-3 & -1 & 1 & -3 & -4 \\
-1 & 1 & 1 & 1 & 0 \\
-3 & 1 & -1 & -3 & -4 \\ 
4 & 0 & 0 & 4 & 6
\end{array}\right).
\end{array}$$

Defining relators for the above set of 12 generators for $\Gamma_1$ are as follows:
$$\begin{array}{llll}
 g_3\, g_{16}\, g_4^2 ,&
 g_3\, g_{10}^{-2}\, g_4^{-1} ,& 
 g_2\, g_4\, g_{16}^{-1}\, g_{15}^{-1} ,&
 g_4\, g_{16}\, g_{11}\, g_{14} ,\\
 g_7\, g_{15}^{-1}\, g_{16}^{-2} ,& 
 g_4\, g_{14}^{-1}\, g_{13}\, g_7^{-1} ,&
 g_3^2\, g_7^{-1}\, g_{16}^{-1} ,&
 g_1\, g_7^{-1}\, g_3\, g_{10} ,\\
 g_4\, g_{10}^{-1}\, g_{14}^{-2} ,&
 g_3\, g_{15}^{-1}\, g_6^{-1}\, g_{10}^{-1} ,&
 g_1\, g_{13}^{-1}\, g_{14}^{-1}\, g_3 ,&
 g_6^2\, g_{10}^{-1}\, g_{14} ,\\
 g_2\, g_{11}\, g_4\, g_6^{-1} ,&
 g_1\, g_{15}^{-1}\, g_{16}\, g_6^{-1} ,&
 g_2\, g_{16}^{-1}\, g_{10}^{-1}\, g_6 ,&
 g_1\, g_{13}\, g_{10}\, g_{15} ,\\
 g_6\, g_{14}^{-1}\, g_{13}^{-2} ,&
 g_1\, g_{11}^2\, g_2^{-1} ,&
 g_2\, g_{14}^{-1}\, g_7^{-1}\, g_{11}^{-1} ,&
 g_3\, g_{13}^{-1}\, g_{11}^{-1}\, g_7 ,\\
 g_1^2\, g_2\, g_{13} ,&
 g_2^2\, g_6^{-1}\, g_{13}^{-1} ,&
 g_7^2\, g_{11}^{-1}\, g_{15} ,&
 g_1\, g_{11}^{-1}\, g_{15}^{-2} .\\
\end{array}$$

\section{Presentation for the Group $\Gamma_2$} 

The 12 side-pairing transformations 
$g_1, g_2, g_3, g_4, g_5, g_8, g_9, g_{12}, g_{13}, g_{14}, g_{15}, g_{16}$ 
that generate the group $\Gamma_2$ are represented 
in $\mathrm{O}^+(4,1)$ as follows:

$$\begin{array}{ll}
g_1 = \frac{1}{2}\left(\begin{array}{rrrrr}
3 & 3 & 1 & 1 & -4 \\
1 & -1 & 1 & -1 & 0 \\
-3 & -3 & 1 & 1 & 4 \\
1 & -1 & -1 & 1 & 0 \\ 
-4 & -4 & 0 & 0 & 6
\end{array}\right), &
g_2 = \frac{1}{2}\left(\begin{array}{rrrrr}
-1 & -1 & -1 & 1 & 0 \\
-3 & 3 & 1 & 1 & -4 \\
3 & -3 & 1 & 1 & 4 \\
1 & 1 & -1 & 1 & 0 \\ 
4 & -4 & 0 & 0 & 6
\end{array}\right), \\ \\
g_3 = \frac{1}{2}\left(\begin{array}{rrrrr}
-1 & -1 & -1 & 1 & 0 \\
-3 & 3 & 1 & 1 & 4 \\
3 & -3 & 1 & 1 & -4 \\
1 & 1 & -1 & 1 & 0 \\ 
-4 & 4 & 0 & 0 & 6
\end{array}\right), &
g_4 = \frac{1}{2}\left(\begin{array}{rrrrr}
3 & 3 & 1 & 1 & 4 \\
1 & -1 & 1 & -1 & 0 \\
-3 & -3 & 1 & 1 & -4 \\
1 & -1 & -1 & 1 & 0 \\ 
4 & 4 & 0 & 0 & 6
\end{array}\right), \\ \\
g_5 = \frac{1}{2}\left(\begin{array}{rrrrr}
1 & -1 & -1 & -1 & 0 \\
1 & 1 & -1 & 1 & 0 \\
3 & -1 & 3 & 1 & -4 \\
3 & 1 & 3 & -1 & -4 \\ 
-4 & 0 & -4 & 0 & 6
\end{array}\right), & 
g_8 = \frac{1}{2}\left(\begin{array}{rrrrr}
1 & -1 & -1 & -1 & 0 \\
1 & 1 & -1 & 1 & 0 \\
3 & -1 & 3 & 1 & 4 \\
3 & 1 & 3 & -1 & 4 \\ 
4 & 0 & 4 & 0 & 6
\end{array}\right), \\ \\
g_9 = \frac{1}{2}\left(\begin{array}{rrrrr}
-1 & -1 & 1 & -1 & 0 \\
1 & 3 & 3 & -1 & -4 \\
1 & -1 & 1 & 1 & 0 \\
-1 & 3 & 3 & 1 & -4 \\ 
0 & -4 & -4 & 0 & 6
\end{array}\right), & 
g_{12} = \frac{1}{2}\left(\begin{array}{rrrrr}
-1 & -1 & 1 & -1 & 0 \\
1 & 3 & 3 & -1 & 4 \\
1 & -1 & 1 & 1 & 0 \\
-1 & 3 & 3 & 1 & 4 \\ 
0 & 4 & 4 & 0 & 6
\end{array}\right), \\ \\ 
g_{13} = \frac{1}{2}\left(\begin{array}{rrrrr}
1 & -1 & 1 & -1 & 0 \\
1 & 1 & -1 & -1 & 0 \\
-3 & -1 & -1 & -3 & 4 \\
3 & -1 & -1 & 3 & -4 \\ 
-4 & 0 & 0 & -4 & 6
\end{array}\right), &
g_{14} = \frac{1}{2}\left(\begin{array}{rrrrr}
1 & 1 & -1 & 1 & 0 \\
3 & 1 & 1 & -3 & 4 \\
-1 & 1 & -1 & -1 & 0 \\
-3 & 1 & 1 & 3 & -4 \\ 
4 & 0 & 0 & -4 & 6
\end{array}\right), \\ \\ 
g_{15} = \frac{1}{2}\left(\begin{array}{rrrrr}
1 & 1 & -1 & 1 & 0 \\
3 & 1 & 1 & -3 & -4 \\
-1 & 1 & -1 & -1 & 0 \\
-3 & 1 & 1 & 3 & 4 \\ 
-4 & 0 & 0 & 4 & 6
\end{array}\right), &
g_{16} = \frac{1}{2}\left(\begin{array}{rrrrr}
1 & -1 & 1 & -1 & 0 \\
1 & 1 & -1 & -1 & 0 \\
-3 & -1 & -1 & -3 & -4 \\
3 & -1 & -1 & 3 & 4 \\ 
4 & 0 & 0 & 4 & 6
\end{array}\right).
\end{array}
$$

Defining relators for the above set of 12 generators for $\Gamma_2$ are as follows:
$$\begin{array}{llll}
 g_2\, g_{12}\, g_4\, g_8^{-1} ,&
 g_4\, g_{12}^{-1}\, g_{13}^{-1}\, g_{16}^{-1} ,&
 g_5\, g_8\, g_{12}^{-1}\, g_{16} ,&
 g_4\, g_{16}^{-1}\, g_{15}\, g_5^{-1} ,\\
 g_3\, g_{12}^{-1}\, g_9^{-1}\, g_4^{-1} ,&
 g_1\, g_9^{-1}\, g_{16}^{-1}\, g_{13}^{-1} ,&
 g_8\, g_{16}^{-1}\, g_{14}^{-1}\, g_{15}^{-1} ,&
 g_5\, g_9^{-1}\, g_{13}\, g_8 ,\\
 g_1\, g_{12}\, g_9\, g_2^{-1} ,&
 g_3\, g_{13}^{-1}\, g_5^{-1}\, g_{12}^{-1} ,&
 g_2\, g_{14}^{-1}\, g_9^{-1}\, g_5 ,&
 g_3\, g_4\, g_{14}^{-1}\, g_{16}^{-1} ,\\
 g_4\, g_{15}\, g_9\, g_{13} ,&
 g_5\, g_{13}^{-1}\, g_{15}^{-1}\, g_{14}^{-1} ,&
 g_1\, g_{13}^{-1}\, g_{14}\, g_8^{-1} ,&
 g_2\, g_{16}^{-1}\, g_8^{-1}\, g_9^{-1} ,\\
 g_3\, g_{15}^{-1}\, g_{12}^{-1}\, g_8 ,&
 g_1\, g_{14}\, g_{12}\, g_{16} ,&
 g_1\, g_3\, g_{15}\, g_4 ,&
 g_1\, g_4\, g_2\, g_{14} ,\\
 g_1\, g_{15}^{-1}\, g_{13}^{-1}\, g_2 ,&
 g_2\, g_8^{-1}\, g_{14}^{-1}\, g_3 ,&
 g_2\, g_3\, g_5^{-1}\, g_{15}^{-1} ,&
 g_1\, g_5^{-1}\, g_3\, g_9 .\\
\end{array}$$

\section{Presentation for the Group $\Gamma_3$} 

The 12 side-pairing transformations 
$g_1, g_2, g_3, g_4, g_6, g_7, g_9, g_{10}, g_{11}, g_{12}, g_{14}, g_{17}$
that generate the group $\Gamma_3$ are represented in $\mathrm{O}^+(4,1)$ as follows:

$$\begin{array}{ll}
g_1 = \frac{1}{2}\left(\begin{array}{rrrrr}
-3 & -3 & -1 & -1 & 4 \\
-1 & 1 & -1 & 1 & 0 \\
-3 & -3 & 1 & 1 & 4 \\
1 & -1 & -1 & 1 & 0 \\ 
-4 & -4 & 0 & 0 & 6
\end{array}\right), &
g_2 = \frac{1}{2}\left(\begin{array}{rrrrr}
3 & -3 & 1 & 1 & 4 \\
-1 & -1 & 1 & -1 & 0 \\
-1 & -1 & -1 & 1 & 0 \\
3 & -3 & -1 & -1 & 4 \\ 
4 & -4 & 0 & 0 & 6
\end{array}\right), \\ \\
g_3 = \frac{1}{2}\left(\begin{array}{rrrrr}
3 & -3 & -1 & 1 & -4 \\
1 & 1 & -1 & -1 & 0 \\
1 & 1 & 1 & 1 & 0 \\
3 & -3 & 1 & -1 & -4 \\ 
-4 & 4 & 0 & 0 & 6
\end{array}\right), &
g_4 = \frac{1}{2}\left(\begin{array}{rrrrr}
-3 & -3 & -1 & 1 & -4 \\
1 & -1 & -1 & -1 & 0 \\
-3 & -3 & 1 & -1 & -4 \\
-1 & 1 & -1 & -1 & 0 \\ 
4 & 4 & 0 & 0 & 6
\end{array}\right), \\ \\
g_6 = \frac{1}{2}\left(\begin{array}{rrrrr}
1 & -1 & 1 & 1 & 0 \\
-1 & -1 & -1 & 1 & 0 \\
3 & 1 & -3 & 1 & 4 \\
3 & -1 & -3 & -1 & 4 \\ 
4 & 0 & -4 & 0 & 6
\end{array}\right), &
g_7 = \frac{1}{2}\left(\begin{array}{rrrrr}
1 & 1 & 1 & -1 & 0 \\
-1 & 1 & -1 & -1 & 0 \\
-3 & -1 & 3 & -1 & 4 \\
3 & -1 & -3 & -1 & -4 \\ 
-4 & 0 & 4 & 0 & 6
\end{array}\right), \\ \\
g_9 = \frac{1}{2}\left(\begin{array}{rrrrr}
-1 & -1 & 1 & -1 & 0 \\
1 & 3 & 3 & -1 & -4 \\
-1 & 1 & -1 & -1 & 0 \\
1 & -3 & -3 & -1 & 4 \\ 
0 & -4 & -4 & 0 & 6
\end{array}\right), & 
g_{10} = \frac{1}{2}\left(\begin{array}{rrrrr}
1 & -1 & -1 & -1 & 0 \\
1 & -3 & 3 & 1 & -4 \\
-1 & -1 & -1 & 1 & 0 \\
-1 & -3 & 3 & -1 & -4 \\ 
0 & 4 & -4 & 0 & 6
\end{array}\right), \\ \\ 
g_{11} = \frac{1}{2}\left(\begin{array}{rrrrr}
-1 & 1 & 1 & -1 & 0 \\
-1 & -1 & -1 & -1 & 0 \\
1 & 3 & -3 & -1 & -4 \\
-1 & 3 & -3 & 1 & -4 \\ 
0 & -4 & 4 & 0 & 6
\end{array}\right), & 
g_{12} = \left(\begin{array}{rrrrr}
0 & 0 & 0 & 1 & 0 \\
0 & 2 & 1 & 0 & 2 \\
1 & 0 & 0 & 0 & 0 \\
0 & -1 & -2 & 0 & -2 \\ 
0 & 2 & 2 & 0 & 3
\end{array}\right), \\ \\ 
g_{14} = \left(\begin{array}{rrrrr}
2 & 0 & 0 & -1 & 2 \\
0 & 0 & 1 & 0 & 0 \\
0 & 1 & 0 & 0 & 0 \\
-1 & 0 & 0 & 2 & -2 \\ 
2 & 0 & 0 & -2 & 3
\end{array}\right), &
g_{17} = \frac{1}{2}\left(\begin{array}{rrrrr}
-1 & 1 & -1 & -1 & 0 \\
1 & 1 & 1 & -1 & 0 \\
1 & 3 & -1 & 3 & -4 \\
1 & -3 & -1 & -3 & 4 \\ 
0 & -4 & 0 & -4 & 6
\end{array}\right). 
\end{array}$$

Defining relators for the above set of 12 generators for $\Gamma_3$ are as follows:
$$\begin{array}{llll}
 g_4^2\, g_6^{-1}\, g_{17}^{-1} ,&
 g_4\, g_{12}^{-1}\, g_{11}^2 ,&
 g_4\, g_{14}^{-1}\, g_{11}\, g_{12} ,&
 g_3\, g_{10}^{-1}\, g_{14}\, g_4 ,\\
 g_2\, g_4\, g_{10}\, g_6^{-1} ,&
 g_3\, g_9\, g_6^{-1}\, g_{10}^{-1} ,&
 g_3\, g_{10}^2\, g_4^{-1} ,&
 g_4\, g_9\, g_{17}^{-1}\, g_{11}^{-1} ,\\
 g_3\, g_{12}^{-1}\, g_7\, g_{11}^{-1} ,&
 g_3\, g_7^{-1}\, g_{10}^{-1}\, g_{12}^{-1} ,&
 g_2\, g_{12}\, g_{11}\, g_{17}^{-1} ,&
 g_1^2\, g_{11}^{-1}\, g_{10} ,\\
 g_1\, g_9^{-1}\, g_{14}^{-1}\, g_7^{-1} ,&
 g_2\, g_3\, g_{14}\, g_{17} ,&
 g_6\, g_{14}^{-1}\, g_{12}\, g_9 ,&
 g_1\, g_{12}\, g_{17}\, g_{14}^{-1} ,\\
 g_1\, g_{17}^{-2}\, g_9 ,&
 g_1\, g_{14}\, g_6\, g_3 ,&
 g_6\, g_7\, g_{12}^{-1}\, g_9^{-1} ,&
 g_2\, g_{17}^{-1}\, g_7^{-1}\, g_3 ,\\
 g_2\, g_9^{-1}\, g_{10}^{-1}\, g_7 ,&
 g_1\, g_2\, g_6^{-2} ,&
 g_2\, g_{11}^{-1}\, g_7^{-1}\, g_9^{-1} ,&
 g_1\, g_7^{-1}\, g_{14}\, g_2^{-1} .\\
\end{array}$$

\section{Presentation for the Group $\Gamma_4$} 

The 12 side-pairing transformations 
$g_1, g_2, g_3, g_4, g_5, g_6, g_8, g_{10}, g_{11}, g_{13}, g_{14}, g_{17}$ 
that generate the group $\Gamma_4$ are represented in $\mathrm{O}^+(4,1)$ as follows:

$$\begin{array}{ll}
g_1 = \frac{1}{2}\left(\begin{array}{rrrrr}
1 & -1 & -1 & 1 & 0 \\
-3 & -3 & 1 & 1 & 4 \\
-3 & -3 & -1 & -1 & 4 \\
1 & -1 & 1 & -1 & 0 \\ 
-4 & -4 & 0 & 0 & 6
\end{array}\right), &
g_2 = \left(\begin{array}{rrrrr}
0 & 0 & -1 & 0 & 0 \\
1 & -2 & 0 & 0 & 2 \\
0 & 0 & 0 & 1 & 0 \\
-2 & 1 & 0 & 0 & -2 \\ 
2 & -2 & 0 & 0 & 3
\end{array}\right), \\ \\
g_3 = \left(\begin{array}{rrrrr}
0 & 0 & -1 & 0 & 0 \\
1 & -2 & 0 & 0 & -2 \\
0 & 0 & 0 & 1 & 0 \\
-2 & 1 & 0 & 0 & 2 \\ 
-2 & 2 & 0 & 0 & 3
\end{array}\right), &
g_4 = \frac{1}{2}\left(\begin{array}{rrrrr}
1 & -1 & -1 & 1 & 0 \\
-3 & -3 & 1 & 1 & -4 \\
-3 & -3 & -1 & -1 & -4 \\
1 & -1 & 1 & -1 & 0 \\ 
4 & 4 & 0 & 0 & 6
\end{array}\right), \\ \\ 
g_5 = \left(\begin{array}{rrrrr}
0 & 1 & 0 & 0 & 0 \\
0 & 0 & 0 & 1 & 0 \\
-1 & 0 & -2 & 0 & 2 \\
2 & 0 & 1 & 0 & -2 \\ 
-2 & 0 & -2 & 0 & 3
\end{array}\right), & 
g_6 = \frac{1}{2}\left(\begin{array}{rrrrr}
3 & -1 & -3 & 1 & 4 \\
-1 & 1 & -1 & 1 & 0 \\
-3 & -1 & 3 & 1 & -4 \\
1 & 1 & 1 & 1 & 0 \\ 
4 & 0 & -4 & 0 & 6
\end{array}\right), \\ \\ 
g_8 = \left(\begin{array}{rrrrr}
0 & 1 & 0 & 0 & 0 \\
0 & 0 & 0 & 1 & 0 \\
-1 & 0 & -2 & 0 & -2 \\
2 & 0 & 1 & 0 & 2 \\ 
2 & 0 & 2 & 0 & 3
\end{array}\right), & 
g_{10} = \frac{1}{2}\left(\begin{array}{rrrrr}
-1 & -1 & -1 & -1 & 0 \\
1 & -1 & -1 & 1 & 0 \\
1 & 3 & -3 & -1 & 4 \\
-1 & 3 & -3 & 1 & 4 \\ 
0 & 4 & -4 & 0 & 6
\end{array}\right), \\ \\ 
g_{11} = \frac{1}{2}\left(\begin{array}{rrrrr}
-1 & -1 & -1 & -1 & 0 \\
1 & -1 & -1 & 1 & 0 \\
1 & 3 & -3 & -1 & -4 \\
-1 & 3 & -3 & 1 & -4 \\ 
0 & -4 & 4 & 0 & 6
\end{array}\right), & 
g_{13} = \left(\begin{array}{rrrrr}
2 & 0 & 0 & 1 & -2 \\
0 & 1 & 0 & 0 & 0 \\
0 & 0 & -1 & 0 & 0 \\
1 & 0 & 0 & 2 & -2 \\ 
-2 & 0 & 0 & -2 & 3
\end{array}\right), \\ \\ 
g_{14} = \left(\begin{array}{rrrrr}
2 & 0 & 0 & -1 & 2 \\
0 & -1 & 0 & 0 & 0 \\
0 & 0 & 1 & 0 & 0 \\
-1 & 0 & 0 & 2 & -2 \\ 
2 & 0 & 0 & -2 & 3
\end{array}\right), &
g_{17} = \frac{1}{2}\left(\begin{array}{rrrrr}
1 & 1 & -1 & -1 & 0 \\
1 & 3 & 1 & 3 & -4 \\
-1 & 1 & 1 & -1 & 0 \\
-1 & 3 & -1 & 3 & -4 \\ 
0 & -4 & 0 & -4 & 6
\end{array}\right). 
\end{array}$$

Defining relators for the above set of 12 generators for $\Gamma_4$ are as follows:
$$\begin{array}{llll}
 g_4\, g_8^{-2}\, g_{17}^{-1} ,&
 g_3\, g_4^2\, g_{13} ,&
 g_4\, g_6^{-1}\, g_{14}\, g_8 ,&
 g_3\, g_{14}\, g_{17}^{-1}\, g_4^{-1} ,\\
 g_3\, g_{13}^{-1}\, g_{17}\, g_{10} ,&
 g_4\, g_{10}^{-1}\, g_{13}^{-1}\, g_8^{-1} ,&
 g_3\, g_{17}\, g_{11}^{-1}\, g_8^{-1} ,&
 g_1\, g_{11}^{-1}\, g_{13}\, g_5^{-1} ,\\
 g_2\, g_{11}^{-1}\, g_4\, g_{14}^{-1} ,&
 g_2\, g_{13}\, g_{17}^{-1}\, g_{11} ,&
 g_3\, g_5^{-1}\, g_6^{-1}\, g_4 ,&
 g_1\, g_{14}\, g_3\, g_{10}^{-1} ,\\
 g_8\, g_{14}^{-1}\, g_{11}^2 ,&
 g_3^2\, g_6\, g_{10}^{-1} ,&
 g_6\, g_{13}\, g_8\, g_{11}^{-1} ,&
 g_5\, g_{10}^{-1}\, g_6^{-1}\, g_{13}^{-1} ,\\
 g_2\, g_{17}^{-1}\, g_{10}^{-1}\, g_5^{-1} ,&
 g_5\, g_{14}\, g_{10}^2 ,&
 g_1\, g_5^{-2}\, g_{17} ,&
 g_1\, g_6\, g_{14}^{-1}\, g_5 ,\\
 g_1\, g_2\, g_8^{-1}\, g_6 ,&
 g_2^2\, g_6^{-1}\, g_{11}^{-1} ,&
 g_1^2\, g_{13}^{-1}\, g_2 ,&
 g_1\, g_{17}^{-1}\, g_{14}\, g_2^{-1} .\\
\end{array}$$


\section{Volumes of Maximum Cusps} 

In this section, we determine the volume of the maximum cusp of each of the one-cusped hyperbolic 24-cell manifolds. 

\begin{proposition} 
The volume of the maximum cusp of each of the one-cusped hyperbolic 24-cell manifolds is 8. 
\end{proposition}
\begin{proof}
We pass to the conformal ball model of $H^4$ and center the 24-cell $Q$ at the origin. 
Two horoballs based at adjacent vertices of $Q$ that project to the maximum cusp 
are tangent at the Euclidean midpoint of the edge of $Q$ joining the vertices.  
From this observation, it is easy to work out the volume of the maximum cusp. 
\end{proof}
%

\section{Orders of Isometry Groups} 

In this section, we determine the order of the isometry group of each of the one-cusped hyperbolic 24-cell manifolds. 

\begin{proposition} 
The order of the group of isometries of the hyperbolic 4-manifold $24c1.k$, for $k = 1, \ldots, 4$,  
is $12, 12, 2, 8$, respectively. 
\end{proposition}
\begin{proof}
The number of side-pairings of the 24-cell $Q$ for the manifold 24c$1.k$, for $k = 1, \ldots, 4$, 
that are equivalent up to a symmetry of $Q$ is $12, 12, 2, 8$ respectively. 
Hence, the order of the group of isometries of the hyperbolic manifold 24c$1.k$, for $k = 1, \ldots, 4$, that are induced by a symmetry of the 24-cell $Q$ is $12, 12, 2, 8$ respectively.  That these are the orders of the full groups of isometries of the manifolds will follow once we prove that $\mathcal{Q}$ is the Epstein-Penner  canonical tessellation \cite{E-P} of $H^4$ determined by the manifolds. 

Let $M= H^4/\Gamma_\ast$ be one of the hyperbolic manifolds 24c1.$k$, for $k = 1, \ldots, 4$.  
Let $C$ be a cusp of $M$.  Choose an ideal vertex $u$ of $Q$. 
Then $C$ is covered by a horoball $B$ based at $u$. 
Let $v$ be the vector on the positive light cone $L^+$ such that the horosphere $\partial B$ 
has the equation $x \circ v = -1$. 
The vector $v$ lies on the ray from the origin in $L^+$ corresponding to $u$.

Each ideal vertex of $Q$ is equivalent to $u$ by the composition of a finite sequence of side-pairing transformations, 
since $M$ has a single cusp. 
Suppose that $u$ is equivalent to the ideal vertex $u'$ be a side-pairing transformation $g$. 
Then $g$ is the composition $rf$ where $f$ is a symmetry of $Q$ that maps $u$ to $u'$ and $r$ is a reflection 
that fixes $u'$. 
Therefore, the orbit $\Gamma_\ast v$ contains the vectors $v = v_1, \ldots, v_{24}$, of the same height on $L^+$,  
on the rays in $L^+$ corresponding to the ideal vertices of $Q$. 
Therefore, the convex hull of the vectors $v = v_1, \ldots, v_{24}$ is a horizontal Euclidean regular 24-cell $F$. 
The remaining vectors in the orbit $\Gamma_\ast v$ are higher up on $L^+$ than $v$, 
since a nonidentity element of $\Gamma_\ast$ moves $Q$ higher up on $H^4$. 
Therefore $F$ is a face of the convex hull of the orbit $\Gamma_\ast v$. 

The face $F$ radially projects from the origin onto $Q$. As $Q$ is a fundamental polytope for $\Gamma_\ast$, 
we deduce that $F$ is a fundamental polytope for the action of $\Gamma_\ast$ on the boundary 
of the convex hull of $\Gamma_\ast v$. 
Therefore $\mathcal{Q}$ is the Epstein-Penner canonical tessellation of $H^4$ determined by $M$. 
Hence, every isometry of $M$ lifts to a symmetry of $\mathcal{Q}$. 
Therefore, every isometry of $M$ is induced by a symmetry of $Q$. 
\end{proof}

\section{Orientable Double Covers} 

The orientable double cover of the flat 3-manifold $N^3_2$ is the 3-torus. 
Therefore, the orientable double cover of each of the one-cusped hyperbolic 24-cell manifolds is a hyperbolic 4-manifold 
with one cusp of link type the 3-torus. 
The homology groups of the orientable double covers of the one-cusped 24-cell manifolds are given in Table \ref{Ta2}. 

Observe that manifolds 24cdc1.1 and 24cdc1.2 have the same homology groups.  This suggests that 
24cdc1.1 and 24cdc1.2 are isometric manifolds.  In fact, these manifolds are isometric, 
since the side-pairings of two 24-cells that glue up the manifolds are equivalent up to symmetries 
of each of the two 24-cells. 

\begin{table}  
\begin{tabular}{llllll}
Name & $H_0$ & $H_1$ & $H_2$ & $H_3$ &  $H_4$ \\
\hline 
24cdc1.1  & $\integers$ & $\integers\oplus\integers_3\oplus\integers_{13}$  & $\integers^2\oplus\integers_3\oplus\integers_{13}$ & $0$  & $0$ \\ 
24cdc1.2  & $\integers$ & $\integers\oplus\integers_3\oplus\integers_{13}$  & $\integers^2\oplus\integers_3\oplus\integers_{13}$ & $0$  & $0$ \\ 
24cdc1.3  & $\integers$ & $\integers\oplus\integers_2^2\oplus \integers_5$ & $\integers^2\oplus\integers_2\oplus\integers_5$ & $0$ & $0$  \\
24cdc1.4  & $\integers$ & $\integers\oplus\integers_3^3$  & $\integers^2\oplus\integers_3^3$ & $0$  & $0$ \end{tabular}

\medskip
\caption{Homology groups of the orientable double covers}\label{Ta2}
\end{table}


\begin{thebibliography}{99}

\bibitem{B-Z} H. Brown, R. B\"ulow, J. Neub\"user, H. Wondratschek, H. Zassenhaus, 
{\it Crystallographic groups of four-dimensional space}, 
John Wiley \& Sons, New York, 1978. 

\bibitem{E-P} D. B. A. Epstein and R. C. Penner, 
Euclidean decompositions of noncompact hyperbolic manifolds, 
{\it J. Differ. Geom.} {\bf 27} (1988) 67-80. 

\bibitem{H-W} W. Hantzsche and H. Wendt, 
Dreidimensionale euklidische Raumformen, 
{\it Math. Ann.} {\bf 110} (1935), 593-611. 


\bibitem{K-M} A. Kolpakov and B. Martelli, 
Hyperbolic four-manifolds with one cusp, 
{\it Geom. Funct. Anal.} {\bf 23} (2013), 1903-1933. 

\bibitem {K-S-S} A. Kolpakov and L. Slavich, 
Symmetries of Hyperbolic 4-manifolds, 
{\it Int. Math. Res. Not. IMRN} {\bf 2016} (2016), 2677-2716. 

\bibitem{K-S-C} A. Kolpakov and L. Slavich, 
Hyperbolic 4-manifolds, colourings and mutations, 
{\it Proc. London. Math. Soc.} {\bf 113} (2016), 163-184. 



\bibitem{R} J. G. Ratcliffe, {\it Foundations of Hyperbolic Manifolds, Third Edition}, 
Graduate Texts in Math., vol. {\bf 149}, Springer Nature Switzerland AG, 2019.

\bibitem{R-T-V} J. G. Ratcliffe and S. T. Tschantz, 
The volume spectrum of hyperbolic 4-manifolds, 
{\it Experimental Math.} {\bf 9} (2000), 101-125. 

\bibitem{R-T-F} J. G. Ratcliffe and S. T. Tschantz, 
Fibered orbifolds and crystallographic groups, 
{\it Algebr. Geom. Topol.} {\bf 10} (2010), 1627-1664. 




\bibitem{S} L. Slavich, 
Some hyperbolic 4-manifolds with low volume and number of cusps, 
{\it Topology Appl.} {\bf 191} (2015), 1-9. 





\end{thebibliography}
\end{document}